\documentclass[a4,12pt]{amsart}
\usepackage{amssymb,amstext,amsmath,amscd,amsthm,
amsfonts,enumerate,graphicx,latexsym}
\usepackage[usenames]{color}

\newtheorem{theorem}{Theorem}[section]
\newtheorem{corollary}[theorem]{Corollary}
\newtheorem{lemma}[theorem]{Lemma}
\newtheorem{proposition}[theorem]{Proposition}

\theoremstyle{definition}
\newtheorem{definition}[theorem]{Definition}

\newtheorem{question}[theorem]{Question}
\newtheorem*{question*}{Question}
\newtheorem{conjecture}[theorem]{Conjecture}
\newtheorem*{conjecture*}{Conjecture}
\newtheorem{example}[theorem]{Example}
\newtheorem{notation}[theorem]{Notation}
\newtheorem*{notation*}{Notation}

\newtheorem*{claim*}{Claim}

\numberwithin{equation}{theorem}
\def\Ext{\operatorname{Ext}}
\def\Tor{\operatorname{Tor}}
\def\rad{\operatorname{rad}}
\def\Hom{\operatorname{Hom}}
\def\sHom{\operatorname{\underline{Hom}}}
\def\End{\operatorname{End}}
\def\rhom{\operatorname{\mathbf{R}Hom}}
\def\CM{\operatorname{\mathsf{CM}}}
\def\sCM{\operatorname{\underline{\mathsf{CM}}}}
\def\d{\operatorname{\mathsf{D}}}
\def\Db{\mathsf{D}^{\rm b}}
\def\k{\operatorname{\mathsf{K}}}
\def\thick{\operatorname{\mathsf{thick}}}
\def\add{\operatorname{\mathsf{add}}}
\def\ind{\operatorname{\mathsf{ind}}}
\def\da{\mathsf{D}(\mathcal{A})}

\def\Mod{\operatorname{\mathsf{Mod}}}
\def\fl{\operatorname{\mathsf{fl}}}
\def\mod{\operatorname{\mathsf{mod}}}
\def\smod{\operatorname{\underline{\mathsf{mod}}}}
\def\proj{\operatorname{\mathsf{proj}}}
\def\inj{\operatorname{\mathsf{inj}}}
\def\gldim{\mathrm{gl}.\!\dim}
\def\pd{\operatorname{pd}}
\def\id{\operatorname{id}}
\def\depth{\operatorname{depth}}
\def\A{\mathcal{A}}
\def\C{\mathcal{C}}
\def\G{\mathcal{G}}
\def\D{\mathcal{D}}
\def\T{\mathcal{T}}
\def\U{\mathcal{U}}
\def\X{\mathcal{X}}

\def\M{\mathcal{M}}
\def\F{\mathcal{F}}
\def\sF{\underline{\mathcal{F}}}
\def\p{\mathfrak{p}}
\def\Z{\mathbb{Z}}

\begin{document}
\setlength{\baselineskip}{15pt}
\title{Remarks on dimensions of triangulated categories}
\author{Takuma Aihara}
\address{Department of Mathematics, Tokyo Gakugei University, 4-1-1 Nukuikita-machi, Koganei, Tokyo 184-8501, Japan}
\email{aihara@u-gakugei.ac.jp}
\author{Ryo Takahashi}
\address{Graduate School of Mathematics, Nagoya University, Furocho, Chikusaku, Nagoya, Aichi 464-8602, Japan}
\email{takahashi@math.nagoya-u.ac.jp}
\urladdr{http://www.math.nagoya-u.ac.jp/~takahashi/}
\thanks{TA was partly supported by JSPS Grant-in-Aid for Young Scientists 15K17516. RT was partly supported by JSPS Grant-in-Aid for Scientific Research 16K05098.}

\begin{abstract}
In this paper, we explore when a locally finite triangulated category has dimension zero or finite representation type. We also study generation of derived categories by orthogonal subcategories.
\end{abstract}
\maketitle
\section{Introduction}

In this paper, we mention two remarks regarding the notion of dimensions of triangulated categories, which has been introduced by Rouquier \cite{R}.

The first subject of this paper, which is discussed in Section 2, is to describe ``smallest'' triangulated categories.
We focus on three kinds of smallness of triangulated categories: local finiteness, dimension zero and finite representation type.
It is natural to ask if there exist implications among these conditions.
We give the following answer to this question.

\begin{theorem}[Theorems \ref{Glr}, \ref{main}]
\begin{enumerate}[\rm(1)]
\item
Let $\Lambda$ be an Iwanaga--Gorenstein algebra over a complete local ring with an isolated singularity.
The stable category $\sCM(\Lambda)$ of Cohen--Macaulay modules is locally finite if and only if $\Lambda$ has finite CM representation type.
\item
Let $\T$ be a Krull--Schmidt triangulated category.
If $\T$ is finitely generated and locally finite, then it has dimension zero.
If $\T$ is Ext-finite and has dimension zero, then it is locally finite.
\end{enumerate}
\end{theorem}

\noindent
Applying the first assertion of this theorem, we deduce that for an isolated hypersurface singularity $R$, the stable category $\sCM(R)$ is locally finite if and only if it has dimension zero, if and only if $R$ has finite CM representation type (Corollary \ref{ihs}).

The second subject of this paper, which is discussed in Section 3, is to find a generator $\G$ of a derived category $\D$ such that the dimension of $\D$ with respect to $\G$ in the sense of \cite{ddcm} is as small as possible.
This brings us a lower bound of the dimension of the derived category. 
The main result on this subject is the following.

\begin{theorem}[Theorem \ref{17121}]
Let $\A$ be an abelian category and $\X$ a full subcategory.
Let $M$ be an object of $\A$ admitting an exact sequence with $X_0,\dots,X_n\in\X$
\begin{align*}
&0\to X_n\to \cdots \to X_0 \to N\to M\to0\\
\text{(resp. }&0\to M\to N\to X_0\to\cdots\to X_n\to0\text{)}
\end{align*}
such that the corresponding element in $\Ext_\A^{n+1}(M,X_n)$ (resp. $\Ext_\A^{n+1}(X_n, M)$) is nonzero.
Then $M$ is outside $\langle {}^\perp\X \rangle_{n+1}$ (resp. $\langle \X^\perp \rangle_{n+1}$) in the derived category $\da$ of $\A$.
\end{theorem}

\noindent
This theorem recovers the lower bounds of the derived dimension given by Rouquier \cite[Proposition 7.14]{R}, Krause and Kussin \cite[Lemma 2.4]{KK} and Yoshiwaki \cite[Theorem 1.1]{Y2}.

\section{Dimension zero versus local finiteness}

\begin{notation}
Throughout this section, let $\T$ be a Krull--Schmidt triangulated category.
We use $k$ and $R$ in this section as an algebraically closed field and a commutative noetherian complete local ring, respectively.
Let $\Lambda$ be a \emph{noetherian $R$-algebra}, i.e., there is a ring homomorphism $R$ to $\Lambda$ with the image contained in the center of $\Lambda$ such that $\Lambda$ is a finitely generated $R$-module.
Note that $\Lambda$ is semiperfect as $R$ is complete (see \cite{A}).
\end{notation}

\subsection{Local finiteness}

We recall the definition of locally finite triangulated categories.

\begin{definition}\label{def:LF}
We say that $\T$ is \emph{locally finite} if for each object $Y$ of $\T$,
\begin{enumerate}[(i)]
\item there are only finitely many indecomposable objects $X$ with $\Hom_\T(X,Y)\neq0$, and
\item for every indecomposable object $X$ of $\T$, the right $\End_\T(X)$-module $\Hom_\T(X, Y)$ has finite length.
\end{enumerate}
Note that these are equivalent to the dual conditions; see \cite{A1, K}.
\end{definition}

For an additive category $\C$ we denote by $\underline\C$ its \emph{stable category}, i.e., the quotient of $\C$ by the projective objects.
This is triangulated if $\C$ is Frobenius \cite{H}.
We denote by $\mod\Lambda$ the category of finitely generated (right) $\Lambda$-modules; it is Frobenius (and hence the stable category $\smod\Lambda$ is triangulated) if $\Lambda$ is selfinjective.
Denote by $\Db(\mod\Lambda)$ the bounded derived category of $\mod\Lambda$.
Here are examples of a locally finite triangulated category.

\begin{example}\label{exLF}
Let $\Lambda$ be a finite dimensional $k$-algebra.
\begin{enumerate}[(1)]
\item
The derived category $\Db(\mod\Lambda)$ is locally finite if and only if $\Lambda$ is a piecewise hereditary algebra of finite representation type.
\item
Suppose that $\Lambda$ is a selfinjective algebra.
Then the stable category $\smod\Lambda$ is locally finite if and only if $\Lambda$ is of finite representation type.
\end{enumerate}
\end{example}

A triangulated category $\T$ is called \emph{finitely generated} if it admits a thick generator $T$, that is, if $\T=\thick T$.
Here, $\thick T$ is the smallest thick subcategory of $\T$ containing $T$.
Important properties of locally finite triangulated categories are stated in \cite{K}, including:

\begin{proposition}\cite[Proposition 4.5]{K}\label{bpLF}
Let $\T$ be a finitely generated locally finite triangulated category.
Then there are only finitely many thick subcategories of $\T$.
\end{proposition}

In representation theory of rings, the notion of representation types is one of the most classical subjects, and the first step is to understand finite representation type.
We study the relationship of local finiteness with finiteness of representation type.
For an additive category $\C$, $\ind\C$ denotes the set of isomorphism classes of indecomposable objects of $\C$.

\begin{proposition}\label{LFRF}
Let $\F$ be a Krull--Schmidt Frobenius category whose stable category $\sF$ satisfies Definition \ref{def:LF}(ii).
Assume that $\proj\F=\{X\in \F\ |\ \Ext_\F^1(M,X)=0 \}$ for some $M\in\F$.
Then $\sF$ is locally finite if and only if it admits an additive generator.
\end{proposition}

\begin{proof}
The `if' part is evident.
For the `only if' part, the local finiteness of $\sF$ implies that
$$
\X=\{X\in\ind\sF\ |\ \sHom_\F(M[-1], X)\ne0\}
$$
is a finite set; we write $\X=\{X_1,\dots,X_n\}$.
Let $Y$ be an indecomposable object of $\F$ which does not belong to $\X$.
Then $\Ext_\F^1(M, Y)$ vanishes.
By assumption $Y$ has to be projective, which implies that $Y=0$ in $\sF$.
Thus, we obtain $\sF=\add(X_1\oplus\cdots\oplus X_n)$.
\end{proof}

For a prime ideal $\p$ of $R$, set $(-)_\p:=-\otimes_RR_\p$.
We say that $\Lambda$ has an \emph{isolated singularity} if $\gldim\Lambda_\p=\dim R_\p$ for any nonmaximal primes $\p$ of $R$ (cf. \cite{IW}).
Denote by $\CM(\Lambda)$ the full subcategory of $\mod\Lambda$ consisting of \emph{Cohen--Macaulay $\Lambda$-modules}, i.e., finitely generated $\Lambda$-modules $X$ with $\Ext_\Lambda^i(X, \Lambda)=0$ for all $i>0$.
We say that $\Lambda$ is \emph{Iwanaga--Gorenstein} if it has finite right and left selfinjective dimension.
An Iwanaga--Gorenstein algebra $\Lambda$ is said to have \emph{finite Cohen--Macaulay (abbr. CM) representation type} if $\CM(\Lambda)$ has an additive generator.
Let us recall a fundamental fact.

\begin{proposition}\label{koremo}
Let $\Lambda$ be an Iwanaga--Gorenstein $R$-algebra with an isolated singularity.
Then the stable category $\sCM(\Lambda)$ is a Krull--Schmidt triangulated category whose Hom-sets have finite length as $R$-modules.
\end{proposition}

\begin{proof}
As is well-known, $\CM(\Lambda)$ is Frobenius, and $\sCM(\Lambda)$ is triangulated \cite{H}.
Moreover, $\CM(\Lambda)$ (and $\sCM(\Lambda)$) inherits the Krull--Schmidt property from $\mod\Lambda$.
Let $M$ and $N$ be in $\CM(\Lambda)$.
Put $d:=\dim R$ and take a nonmaximal prime ideal $\p$ of $R$.
We get isomorphisms
\[\sHom_\Lambda(M, N[d])_\p\simeq \sHom_{\Lambda_\p}(M_\p, N_\p[d])\simeq \Ext_{\Lambda_\p}^d(M_\p, N_\p)=0.\]
Here, the first isomorphism comes from the fact that the functor $(-)_\p$ is exact and the last equality holds since $\gldim \Lambda_\p=\dim R_\p<d$.
Therefore, we see that $\sHom_\Lambda(M, N[d])$ has finite length as an $R$-module,
whence so does any Hom-set of $\sCM(\Lambda)$.
\end{proof}

We deduce the following from Propositions \ref{LFRF} and \ref{koremo}, which extends Example \ref{exLF}(2).

\begin{theorem}\label{Glr}
Let $\Lambda$ be an Iwanaga--Gorenstein $R$-algebra with an isolated singularity.
Then $\sCM(\Lambda)$ is locally finite if and only if $\Lambda$ has finite CM representation type. 
\end{theorem}

\begin{proof}
Put $S:=\Lambda/\rad\Lambda$, where $\rad\Lambda$ stands for the Jacobson radical of $\Lambda$.
As $d:=\id\Lambda_\Lambda<\infty$, we observe that $\Omega^dS$ belongs to $\CM(\Lambda)$.
We show that $\Omega^dS$ plays the role of $M$ as in Proposition \ref{LFRF}.
Let $X\in\CM(\Lambda)$ satisfy $\Ext_\Lambda^1(\Omega^dS, X)=0$.
Then $\Ext_\Lambda^{d+1}(S, X)=0$.
Thanks to Auslander's result, $X$ has finite injective dimension; see \cite[Corollary 3.5(3)]{GN} for instance.
Hence $X$ is projective.
Now Propositions \ref{LFRF} and \ref{koremo} complete the proof.
\end{proof}

\subsection{Dimension zero}

For $X\in\T$ set $\langle X\rangle:=\add\{X[i]\mid i\in\mathbb{Z} \}$.
We say \emph{$\T$ has dimension zero} and write $\dim\T=0$, if $\T=\langle X\rangle$ for some $X\in\T$.
The following are well-known.

\begin{example}\label{exd0}
Let $\Lambda$ be a finite dimensional $k$-algebra.
\begin{enumerate}[(1)]
\item\cite[Theorem]{CYZ}
One has $\dim\Db(\mod\Lambda)=0$ if and only if $\Lambda$ is a piecewise hereditary algebra of finite representation type.
\item\cite[Corollary 3.10]{Y}
Suppose that $\Lambda$ is selfinjective.
Then $\dim(\smod\Lambda)=0$ if and only if $\Lambda$ is of finite representation type.
\end{enumerate}
\end{example}

\noindent
Comparing this example with Example \ref{exLF},
we have a natural question.

\begin{question}\label{question}
Is it true that $\T$ is locally finite if and only if it has dimension zero?
\end{question}

A $k$-linear triangulated category $\T$ is called \emph{Ext-finite} if $\sum_{i\in\mathbb{Z}}\dim_k\Hom_\T(X, Y[i])<\infty$ for any objects $X$ and $Y$ of $\T$.
Here are our answers to Question \ref{question}. 

\begin{theorem}\label{main}
\begin{enumerate}[\rm(1)]
\item If $\T$ is finitely generated and locally finite, then it has dimension $0$.

\item Assume that $\T$ is Ext-finite.
If $\T$ has dimension $0$, then it is locally finite.
\end{enumerate}
\end{theorem}
\begin{proof}
(1) Let $T$ be a thick generator of $\T$ and put $\M:=\{M\in \ind\T\ |\ \Hom_\T(T, M)\neq0\}$.
Let $N$ be an indecomposable object of $\T$.
Since $T$ is a thick generator, we observe that if $\Hom_\T(T, N[i])=0$ for all integers $i$, then $N=0$, contrary to the choice of $N$.
Hence $N[\ell]$ belongs to $\M$ for some integer $\ell$.
This shows that every object of $\T$ is in $\langle \M\rangle$, and thus $\T=\langle \M\rangle$.
As $\T$ is locally finite, $\M$ is a finite set.
Therefore, putting $M$ as the direct sum of all objects in $\M$, one has $\T=\langle M\rangle$, and obtains $\dim\T=0$.

(2) As $\dim\T=0$, there is an object $M\in\T$ with $\T=\langle M \rangle$.
Taking an indecomposable decomposition $M=M_1\oplus M_2\oplus\cdots\oplus M_r$, we obtain $\ind\T=\{M_i[j]\ |\ i=1,\dots,r\ \mbox{and } j\in\mathbb{Z} \}$ by the Krull--Schmidtness of $\T$.
Let $Y\in\T$.
Since $\T$ is Ext-finite, there exist only finitely many indecomposable objects $X$ with $\Hom_\T(X, Y)\neq0$,
whence $\T$ is locally finite.
\end{proof}

The following is an extension of Proposition \ref{bpLF} from the viewpoint of Theorem \ref{main}(1).

\begin{proposition}
If $\dim\T=0$, then $\T$ possesses only fnitely many thick subcategories.
\end{proposition}

\begin{proof}
By assumption, we have $\T=\langle X\rangle$ for some $X\in\T$.
Write $X=X_1\oplus X_2\oplus\cdots\oplus X_n$ with each $X_i$ indecomposable.
Let $\U$ be a thick subcategory of $\T$.
Then we may assume that $X_1,\dots,X_m$ are in $\U$, and $X_{m+1},\dots,X_n$ are not in $\U$.
Now it is evident that $\U=\langle X_1\oplus\cdots\oplus X_m\rangle$, and the assertion follows.
\end{proof}

\subsection{Applications}

In this subsection, let $\Lambda$ be an Iwanaga--Gorenstein $R$-algebra with an isolated singularity.
Here is a conjecture, which seems to be folklore:

\begin{conjecture}\label{conjecture}
The stable category $\sCM(\Lambda)$ has dimension zero if and only if the algebra $\Lambda$ has finite CM representation type.
\end{conjecture}

In view of Theorem \ref{Glr}, we get the following.

\begin{corollary}\label{questionconjecture}
Question \ref{question} is affirmative for $\T=\sCM(\Lambda)$ if and only if Conjecture \ref{conjecture} holds true for $\Lambda$.
\end{corollary}

This gives a positive answer to Questoin \ref{question} in the case where $R$ is a hypersurface.

\begin{corollary}\label{ihs}
The following are equivalent for an isolated hypersurface singularity $R$:
\begin{enumerate}[\rm(i)]
\item $\sCM(R)$ is locally finite;
\item $\sCM(R)$ has dimension zero;
\item $R$ is of finite CM representation type.
\end{enumerate}
In particular, Question \ref{question} is affirmative for $\T=\sCM(R)$.
\end{corollary}

\begin{proof}
It follows from \cite[Propositions 2.4(2) and 2.5]{DT} that Conjecture \ref{conjecture} holds for $\Lambda=R$, whence Question \ref{question} is affirmative for $\T=\sCM(R)$ by Corollary \ref{questionconjecture}.
\end{proof}

\section{Generation by orthogonal subcategories}

\begin{notation}
Throughout this section, let $R$ be a ring (not necessarily noetherian or commutative).
Let $\A$ be an abelian category and $\X$ a full subcategory.
We denote by ${}^\perp\X$ (resp. ${}^{\perp_n}\X$) the full subcategory of $\A$ consisting of objects $M$ such that $\Ext_\A^i(M,\X)=0$ for all $i\ge1$ (resp. $1\le i\le n$).
Dually, we denote by $\X^\perp$ (resp. $\X^{\perp_n}$) the full subcategory of $\A$ consisting of objects $M$ such that $\Ext_\A^i(\X,M)=0$ for all $i\ge1$ (resp. $1\le i\le n$).
\end{notation}

The following is the main result of this section.

\begin{theorem}\label{17121}
Let $\X$ be a full subcategory of $\A$, $M$ an object of $\A$ and $n\ge0$ an integer.
\begin{enumerate}[\rm(1)]
\item
Suppose that there exists an exact sequence in $\A$ of the form
$$
0 \to X_n \to \cdots \to X_0 \to N \to M \to 0
$$
with $X_0,\dots,X_n\in\X$ such that the corresponding element in $\Ext_\A^{n+1}(M,X_n)$ is nonzero.
Then $M\notin{\langle{}^\perp\X\rangle}_{n+1}$ in $\da$.
\item
Suppose that there exists an exact sequence in $\A$ of the form
$$
0 \to M \to N \to X^0 \to \cdots \to X^n \to 0
$$
with $X^0,\dots,X^n\in\X$ such that the corresponding element in $\Ext_\A^{n+1}(X_n,M)$ is nonzero.
Then $M\notin{\langle\X^\perp\rangle}_{n+1}$ in $\da$.
\end{enumerate}
\end{theorem}

\begin{proof}
We only prove the first assertion, because the second one can be shown by a dual argument.
Let $L=L_0$ be the image of the morphism $X_0\to N$, and for each $1\le i\le n$ let $L_i$ be the image of the morphism $X_i\to X_{i-1}$.
Then there is an exact sequence
\begin{equation}\label{17123}
0 \to X_n \to \cdots \to X_i \to L_i \to 0
\end{equation}
for $0\le i\le n$.
Set $L_{-1}=M$ and $X_{-1}=N$.
For each $-1\le i\le n-1$ the exact sequence $0\to L_{i+1}\to X_i\to L_i\to0$ induces a morphism $\delta_i:L_i\to L_{i+1}[1]$ in $\da$, and we obtain the composite morphism
$$
\eta:M[-1]=L_{-1}[-1]\xrightarrow{\delta_{-1}[-1]}L=L_0\xrightarrow{\delta_0}L_1[1]\xrightarrow{\delta_1[1]}\cdots\xrightarrow{\delta_{n-1}[n-1]}L_n[n]=X_n[n]
$$
in $\da$.
The morphism $\eta$ corresponds to the original exact sequence $0 \to X_n \to \cdots \to X_0 \to N \to M \to 0$ via the isomorphism $\Ext_\A^{n+1}(M,X_n)\cong\Hom_{\da}(M[-1],X_n[n])$.
By assumption $\eta$ is nonzero.

Fix any $C\in{}^\perp\X$ and $j\in\Z$.
Using \eqref{17123}, we easily see that $\Ext_\A^k(C,L_i)=0$ for all $0\ne k\in\Z$ and $0\le i\le n$, which implies $\Hom_{\da}(C[j],L_i[h])\cong\Ext_\A^{h-j}(C,L_i)=0$ for all $0\le i\le n$ and $j\ne h\in\Z$.
Note also that $\Hom_{\da}(C[j],L_{-1}[h])\cong\Ext_\A^{h-j}(C,M)=0$ for all $j>h\in\Z$.
Now it is easy to check that the induced map
$$
\Hom_{\da}(C[j],\delta_i[i]):\Hom_{\da}(C[j],L_i[i])\to\Hom_{\da}(C[j],L_{i+1}[i+1])
$$
is zero for $-1\le i\le n-1$ since either the domain or target of the map vanishes.
This shows $\Hom_{\da}({\langle{}^\perp\X\rangle}_1,\delta_i[i])=0$ for all $-1\le i\le n-1$.
Applying the ghost lemma, we deduce $\Hom_{\da}({\langle{}^\perp\X\rangle}_{n+1},\eta)=0$, which implies $M[-1]\notin{\langle{}^\perp\X\rangle}_{n+1}$ as $\eta\ne0$.
\end{proof}

To give our next result, we introduce some notation.

\begin{notation}
Let $\X$ be a full subcategory of $\A$.
We denote by $\X_n$ the full subcategory of $\A$ consisting of objects $M$ such that there exists an exact sequence $0 \to X_n \to \cdots \to X_0 \to M \to 0$ in $\A$ with $X_i\in\X$ for all $0\le i\le n$.
Dually, we denote by $\X^n$ the full subcategory of $\A$ consisting of objects $M$ such that there exists an exact sequence $0 \to M \to X^0 \to \cdots \to X^n \to 0$ in $\A$ with $X^i\in\X$ for all $0\le i\le n$.
\end{notation}

The first assertion of Theorem \ref{17121} immediately recovers \cite[Theorem 1.1]{Y2} as follows.
Here, ${}^\perp T\text{-}\mathrm{tri.dim}\,\Db(\mod R)$ is the (Rouquier) dimension of the triangulated category $\Db(\mod R)$ with respect to the full subcategory ${}^\perp T$ of $\mod R$; see \cite{ddcm}.

\begin{corollary}[Yoshiwaki]
Let $R$ be a noetherian ring and $T$ a cotilting $R$-module of injective dimension $d>0$.
Then ${}^\perp T\text{-}\mathrm{tri.dim}\,\Db(\mod R)=\mathrm{inj.dim}\,T$.
\end{corollary}

\begin{proof}
Choose an $R$-module $M$ with $\Ext_R^d(M,T)\ne0$, and take a Cohen--Macaulay approximation $0\to L\to N\to M\to0$, that is, an exact sequence of $R$-modules such that $N\in{}^\perp T$ and $L\in(\add T)_d$.
Applying Theorem \ref{17121}(1) to $\A=\mod R$, $\X=\add T$ and $n=d-1$, we have $M\notin{\langle{}^\perp T\rangle}_d$.
Combine this with \cite[Theorem 5.3]{ddcm}.
\end{proof}

Our Theorem \ref{17121} also gives rise to the following result.

\begin{corollary}\label{17126}
Let $\X$ be a full subcategory of $\A$, $M$ an object of $\A$ and $n>0$ an integer.
\begin{enumerate}[\rm(1)]
\item
If there is an exact sequence $0\to X_n\to\cdots\to X_0\to M\to0$ which is nonzero in $\Ext_\A^n(M,X_n)$, then $M\notin{\langle{}^\perp\X\rangle}_n$ in $\da$.
\item
If there is an exact sequence $0\to M\to X^0\to\cdots\to X^n\to0$ which is nonzero in $\Ext_\A^n(X^n,M)$, then $M\notin{\langle\X^\perp\rangle}_n$ in $\da$.
\end{enumerate}
\end{corollary}

\begin{proof}
For (1), as $n\ge1$, we can apply Theorem \ref{17121}(1) to the exact sequence $0\to X_n\to\cdots\to X_1\to X_0\to M\to0$ to deduce that $M\notin{\langle{}^\perp\X\rangle}_n$, and (2) is shown dually.
\end{proof}

Let $M$ be an object $M$ of $\A$, and let $n$ be a nonnegative integer.
We say that \emph{$M$ has projective (resp. injective) dimension $n$} and denote it by $\pd M=n$ (resp. $\id M=n$) if there exists an exact sequence $0\to P_n\to P_{n-1}\to\cdots\to P_0\to M\to0$ (resp. $0\to M\to I^0\to\cdots\to I^n\to0$) with the $P_i$ projectives (resp. $I^i$ injectives), and no shorter such sequence exists.
The above result deduces the following corollary, whose partial assertion ``$M\notin{\langle\proj\A\rangle}_{\pd M}$ if $\pd M<\infty$'' is nothing but \cite[Lemma 2.4]{KK}.

\begin{corollary}\label{17124}
For an object $M\in\A$ one has:
$$
M\notin
\begin{cases}
{\langle\proj\A\rangle}_{\pd M} & \text{if $\pd M<\infty$},\\
{\langle\inj\A\rangle}_{\id M} & \text{if $\id M<\infty$},
\end{cases}
\qquad
M\notin
\begin{cases}
{\langle{}^\perp(\proj\A)\rangle}_{\pd M} & \text{if $\pd M<\infty$},\\
{\langle(\inj\A)^\perp\rangle}_{\id M} & \text{if $\id M<\infty$}.
\end{cases}
$$
\end{corollary}

\begin{proof}
The first assertion follows from the second one since $\proj\A$ and $\inj\A$ are contained in ${}^\perp(\proj\A)$ and $(\inj\A)^\perp$, respectively.
In what follows, we show the second assertion.
Setting $\pd M=n<\infty$, we have an exact sequence $\sigma:0\to P_n\xrightarrow{f}P_{n-1}\to\cdots\to P_0\to M\to0$ with $P_i\in\proj\A$ for $0\le i\le n$.
Then the morphism
$$
\begin{CD}
P\ @. =\ @. (0 @>>> P_n @>f>> P_{n-1} @>>> \cdots @>>> P_0 @>>> 0)\\
@V{\eta}VV @. @. @| @VVV @. @VVV\\
P_n[n]\ @. =\ @. (0 @>>> P_n @>>> 0 @>>> \cdots @>>> 0 @>>> 0)
\end{CD}
$$
in $\da$ corresponds to $\sigma$ via the isomorphism $\Hom_{\da}(P,P_n[n])\cong\Ext_\A^n(M,P_n)$.
All the objects of $\A$ appearing in the above diagram are in $\proj\A$, and we can regard $\eta$ as a morphism in $\k(\proj\A)$.
Suppose that $\eta$ is zero in $\da$.
Then $\eta$ is zero in $\k(\proj\A)$, which means that $\eta$ is null-homotopic.
It is easy to observe from this that $f$ is a split monomorphism in $\A$.
This implies that $M$ has a projective resolution of length $n-1$, which contradicts the fact that $\pd M=n$.
Therefore $\eta$ is nonzero in $\da$, and it follows from Corollary \ref{17126} that $M$ in not in ${\langle{}^\perp(\proj\A)\rangle}_n$.
The assertion for $\inj\A$ is shown dually.
\end{proof}

Let $R$ be a commutative ring. A finitely generated $R$-module $C$ is called \emph{semidualizing} if the homothety map $R\to\End_R(C)$ is an isomorphism and $\Ext_R^i(C,C)=0$ for all $i>0$. We give an application of Corollaries \ref{17126} and \ref{17124}, which involves semidualizing modules.

\begin{corollary}\label{17127}
Let $R$ be a commutative noetherian ring, and let $C$ be a semidualizing $R$-module.
The following hold for a (not necessarily finitely generated) $R$-module $M$.
\begin{enumerate}[\rm(1)]
\item
If $\pd_R(\Hom_R(C,M))=n<\infty$, then $M\notin{\langle{}^\perp C\rangle}_n$ in $\d(\Mod R)$.
\item
If $\id_R(C\otimes_RM)=n<\infty$, then $M\notin{\langle C^\perp\rangle}_n$ in $\d(\Mod R)$.
\end{enumerate}
\end{corollary}

\begin{proof}
Let $\sigma:0\to P_n\to\cdots\to P_0\to\Hom_R(C,M)\to0$ be a projective resolution of $\Hom_R(C,M)$ in $\Mod R$.
Tensoring $C$ gives an exact sequence
$$
C\otimes_R\sigma:0\to C\otimes_RP_n\to\cdots\to C\otimes_RP_0\to M\to0
$$
by (the proof of) \cite[Theorem 2.11(c)]{cdim}.
Suppose that $C\otimes_R\sigma$ is zero in $\Ext_R^n(M,C\otimes_RP_n)$.
Then $\sigma=\Hom_R(C,C\otimes_R\sigma)$ is zero in $\Ext_R^n(\Hom_R(C,M),P_n)$, and as we have already seen in the proof of Corollary \ref{17124}, this yields a contradiction.
Hence $C\otimes_R\sigma$ is nonzero in $\Ext_R^n(M,C\otimes_RP_n)$, and we can apply Corollary \ref{17126} for $\X=\add C$ to deduce that $M$ is not in ${\langle{}^\perp C\rangle}_n$.
This shows (1), and a dual argument implies (2).
\end{proof}

Let $R$ be a commutative noetherian local ring.
Recall that a finitely generated $R$-module $M$ is called {\em (maximal) Cohen--Macaulay} if $\depth M\ge\dim R$.
If $R$ is furthermore (Iwanaga--)Gorenstein, then the two definitions of a Cohen--Macaulay module coincide.
We use the same notation $\CM(R)$ as in Section 2 to denote the full subcategory of $\mod R$ consisting of Cohen--Macaulay modules.
From Corollary \ref{17124} or \ref{17127}, we immediately obtain:

\begin{corollary}
Let $R$ be a noetherian ring, and let $M$ be a finitely generated $R$-module.
Let $n$ be a nonnegative integer.
Then the following statements hold in $\Db(\mod R)$.
\begin{enumerate}[\rm(1)]
\item
If $M$ has projective dimension $n$, then $M\notin{\langle{}^\perp R\rangle}_n$.
\item
If $R$ is Iwanaga--Gorenstein and $M$ has projective dimension $n$, then $M\notin{\langle\CM(R)\rangle}_n$.
\item
If $R$ is a commutative Cohen--Macaulay local ring with canonical module $\omega$ and $\Hom_R(\omega,M)$ has projective dimension $n$, then $M\notin{\langle\CM(R)\rangle}_n$.
\end{enumerate}
\end{corollary}

To prove our next result, we establish a lemma which should be well-known to experts.

\begin{lemma}\label{17128}
Let $R$ be a commutative noetherian local ring.
Let $M,N$ be finitely generated $R$-modules with $\Ext_R^{>0}(M,N)=0$.
Then $\depth_R\Hom_R(M,N)=\depth_RN$.
\end{lemma}

\begin{proof}
Let $k$ be the residue field of $R$, and put $t=\depth_RN$.
There are isomorphisms
$$
\rhom_R(k,\Hom_R(M,N))\cong\rhom_R(k,\rhom_R(M,N))\cong\rhom_R(k\otimes_R^\mathbf{L}M,N)
$$
which yield a spectral sequence
$$
E_2^{pq}=\Ext_R^p(\Tor^R_q(k,M),N)\Rightarrow H^{p+q}=\Ext_R^{p+q}(k,\Hom_R(M,N)).
$$
Since $E_2^{pq}=0$ for all $(p,q)$ with $p<t$, we see that $H^t=E_2^{t0}$ and $H^i=0$ for all $i<t$.
Hence $\Ext_R^i(k,\Hom_R(M,N))=0$ for all $i<t$, and $\Ext_R^t(k,\Hom_R(M,N))$ is isomorphic to $\Ext_R^t(k\otimes_RM,N)$.
As $k\otimes_RM$ is a nonzero $k$-vector space and $\Ext_R^t(k,N)$ is nonzero, so is $\Ext_R^t(k\otimes_RM,N)$, and so is $\Ext_R^t(k,\Hom_R(M,N))$.
Thus $\depth_R\Hom_R(M,N)=t$.
\end{proof}

In the case of local rings, the number $n$ in Corollary \ref{17127}(1) can be computed by using the well-established invariant of depth.

\begin{corollary}
Let $R$ be a commutative noetherian local ring, and let $C$ be a semidualizing $R$-module.
Let $M$ be a finitely generated $R$-module such that $\Hom_R(C,M)$ has finite projective dimension.
Then $M\notin{\langle{}^\perp C\rangle}_n$ in $\Db(\mod R)$, where $n=\depth R-\depth M$.
\end{corollary}

\begin{proof}
It follows from \cite[Theorem 2.11(c) and Corollary 2.9(a)]{cdim} that $\Ext_R^{>0}(C,M)=0$, and by Lemma \ref{17128} we have $\depth\Hom_R(C,M)=\depth M$.
Using the Auslander--Buchsbaum formula, we get $\pd\Hom_R(C,M)=\depth R-\depth\Hom_R(C,M)=\depth R-\depth M$.
The assertion now follows from Corollary \ref{17127}(1).
\end{proof}

We also detect relationships among orthogonal subcategories.

\begin{corollary}\label{17122}
Let $\X$ be a full subcategory of $\A$.
One then has:
$$
{}^{\perp_n}\X\cap{\langle{}^\perp\X\rangle}_{n+1}\subseteq{}^{\perp_1}(\X_n),\qquad
\X^{\perp_n}\cap{\langle\X^\perp\rangle}_{n+1}\subseteq{(\X^n)}^{\perp_1}
$$
\end{corollary}

\begin{proof}
We only prove the first equality; the second one is shown dually.
Pick any object $M$ in ${}^{\perp_n}\X\cap{\langle{}^\perp\X\rangle}_{n+1}$.
Assume that $M$ is not in ${}^{\perp_1}(\X_n)$.
Then there exists $N\in\X_n$ such that $\Ext_\A^1(M,N)\ne0$.
This implies that there are exact sequences
$$
\sigma:\ 0 \to N \to E \to M \to 0,\qquad
\tau:\ 0 \to X_n \to \cdots \to X_0 \to N \to 0
$$
in $\A$ such that $\sigma$ is nonsplit and $X_0,\dots,X_n$ belong to $\X$.
We have morphisms
$$
\Hom_\A(M,M)\xrightarrow{f}\Ext_\A^1(M,N)\xrightarrow{g}\Ext_\A^{n+1}(M,X_n),
$$
where $f,g$ are induced from $\sigma,\tau$ respectively.
The map $f$ sends the identity map $\id_M$ of $M$ to the element $[\sigma]\in\Ext_\A^1(M,N)$ corresponding to the short exact sequence $\sigma$.
As $\sigma$ is nonsplit, $[\sigma]$ is nonzero.
Splicing $\sigma$ and $\tau$, we get an exact sequence
$$
\upsilon:\ 0 \to X_n \to \cdots \to X_0 \to E \to M \to 0.
$$
The map $g$ sends $[\sigma]$ to $[\upsilon]\in\Ext_\A^{n+1}(M,X_n)$.
Since $M$ is in ${}^{\perp_n}\X$, it is easy to see that $g$ is injective.
Hence $[\upsilon]$ is also nonzero.
Thus we can apply Theorem \ref{17121}(1) to see that $M$ is not in ${\langle{}^\perp\X\rangle}_{n+1}$, contrary to the choice of $M$.
This contradiction shows the assertion.
\end{proof}

We denote by $\fl R$ the full subcategory of $\mod R$ consisting of modules of finite length.
Here is an application of the above result.

\begin{corollary}\label{17125}
Let $R$ be a commutative Cohen--Macaulay local ring of Krull dimension $d>0$ with canonical module $\omega$.
Then $\fl R\cap{\langle\CM(R)\rangle}_d=0$.
\end{corollary}

\begin{proof}
Note that $\CM(R)={}^\perp\omega$ in the abelian category $\mod R$.
Applying Corollary \ref{17122} to $\X=\add\omega$ and $n=d-1\ge0$, we see that ${}^{\perp_{d-1}}\omega\cap{\langle\CM(R)\rangle}_d$ is contained in ${}^{\perp_1}(({\add\omega})_{d-1})$.

Let $L$ be an $R$-module of finite length that belongs to ${\langle\CM(R)\rangle}_d$.
Then $L$ is in ${}^{\perp_{d-1}}\omega$ since $\omega$ has depth $d$.
Hence $L$ is in ${}^{\perp_1}(({\add\omega})_{d-1})$, that is, $\Ext_R^1(L,({\add\omega})_{d-1})=0$.
Take a minimal Cohen--Macaulay approximation $0 \to B \to A \to L \to 0$.
Then $B$ belongs to $(\add\omega)_d$.
If $L$ is nonzero, then the depth lemma shows that $B$ has depth $1$, and again the depth lemma shows that $B$ is in $(\add\omega)_{d-1}$.
Thus the above exact sequence must split, which is a contradiction.
It follows that $L=0$, and the proof is completed.
\end{proof}

Corollary \ref{17125} immediately recovers \cite[Proposition 7.14]{R}:

\begin{corollary}[Rouquier]
Let $R$ be a commutative noetherian local ring with residue field $k$.
Then $k\notin{\langle R\rangle}_{\dim R}$.
\end{corollary}

\noindent
Indeed, we can reduce to the case where $R$ is regular with $\dim R>0$, and the corollary applies.
One can also recover it from the first assertion of Corollary \ref{17124}.


\end{document}